\documentclass[a4paper,12pt]{amsart}

\usepackage{tikz-cd}

\usepackage[latin1]{inputenc} 
\usepackage{graphicx}
\usepackage[matrix,arrow,tips,curve]{xy}
\usepackage[english]{babel}
\usepackage{amsmath}
\usepackage{amssymb}

\usepackage{mathrsfs}

\usepackage{enumerate,bm}

\newtheorem{thm}{Theorem}[section]
\newtheorem{lemma}[thm]{Lemma}
\newtheorem{corollary}[thm]{Corollary}
\newtheorem{proposition}[thm]{Proposition}

\newtheorem*{thm*}{Theorem}

\theoremstyle{definition}

\newtheorem{remark}[thm]{Remark}

\newcommand{\ph}{\varphi}

\newcommand{\ma}{\mathcal}
\newcommand{\la}{\longrightarrow}
\newcommand{\ol}{\mathcal{O}}

\newcommand{\pr}{\mathbb{P}}

\newcommand{\R}{\mathbb{R}}
\newcommand{\Z}{\mathbb{Z}}
\newcommand{\N}{\mathcal{N}_1}
\newcommand{\Nu}{\mathcal{N}^1}

\newcommand{\depth}{\operatorname{depth}}

\newcommand{\NE}{\operatorname{NE}}
\newcommand{\Exc}{\operatorname{Exc}}

\newcommand{\codim}{\operatorname{codim}}
\newcommand{\Eff}{\operatorname{Eff}}

\newcommand{\mov}{\operatorname{mov}}

\usepackage{etoolbox}
\patchcmd{\section}{\normalfont}{\normalfont\large}{}{}
\patchcmd{\subsection}{\bfseries}{\scshape\centering}{}{}
\patchcmd{\subsection}{-.5em}{.5em}{}{}

\makeatletter
\@namedef{subjclassname@2020}{\textup{2020} Mathematics Subject Classification}
\makeatother

\title[Fano $4$-folds with $b_2>12$ are products]{Fano $4$-folds with $b_2>12$ are products of surfaces}
\author{C.~Casagrande}
\address{Universit\`a di Torino,
Dipartimento di Matematica,
via Carlo Alberto 10,
10123 Torino - Italy}
\email{cinzia.casagrande@unito.it}
\date{November 1, 2023}

\calclayout 

\begin{document}
\maketitle

{\hfill\em\small Dedicated to Lorenzo, Sabrina, and Fabrizio}

\bigskip

\begin{abstract}
Let $X$ be a smooth, complex Fano $4$-fold, and $\rho_X$ its Picard number. We show that if $\rho_X>12$, then $X$ is a product of del Pezzo surfaces. The proof relies on a careful study of divisorial elementary contractions $f\colon X\to Y$ such that $\dim f(\Exc(f))=2$, together with the author's previous work  on Fano $4$-folds. In particular, given $f\colon X\to Y$ as above, under suitable assumptions we show  that $S:=f(\Exc(f))$ is a smooth del Pezzo surface with $-K_S=(-K_Y)_{|S}$.
\end{abstract}
\section{Introduction}
\noindent Smooth, complex Fano varieties 
have been classically intensively studied, and have attracted a lot of attention also in the last decades, due to their role in the framework of the Minimal Model Program. The Fano condition is a natural positivity condition of the tangent bundle, and it ensures a rich geometry, from both the points of view of birational geometry and of families of rational curves.

It has been known since the 90's that Fano varieties form a bounded family in each dimension. Del Pezzo surfaces are known classically, and the classification of Fano $3$-folds have been in achieved in the 80's, there are 105 families.

Starting from dimension $4$, there are probably too many families to get a complete classification; still  we aim to better understand and describe the behavior and properties of these varieties.
In this paper we focus on Fano $4$-folds $X$ with ``large'' Picard number $\rho_X$; let us recall that since $X$ is Fano, $\rho_X$ is equal to the second Betti number $b_2(X)$.
We show the following  result.
 \begin{thm}\label{main}
   Let $X$ be a smooth Fano $4$-fold with $\rho_X> 12$. Then $X\cong S_1\times S_2$, where $S_i$ are del Pezzo surfaces.
\end{thm}
To the author's knowledge, all known examples of Fano $4$-folds which are not products of surfaces have $\rho\leq 9$, so that we do not know whether the condition 
$\rho> 12$ in Theorem \ref{main} is sharp. We refer the  reader to \cite[\S 6]{rendiconti} for an overview of known Fano $4$-folds with $\rho\geq 6$; there are few examples and it is an interesting problem to construct new ones.

As $\rho_{S_1\times S_2}=\rho_{S_1}+\rho_{S_2}$, and del Pezzo surfaces have $\rho\leq 9$, Theorem \ref{main} implies the following.
\begin{corollary}\label{max}
Let $X$ be a smooth Fano $4$-fold. Then $\rho_X\leq 18$.
\end{corollary}  
Let us note that Theorem \ref{main} and Corollary \ref{max} generalize to dimension $4$ the analogous result for Fano $3$-folds, established by Mori and Mukai in the 80's:
\begin{thm}[\cite{morimukai2}, Theorem 1.2]
  Let $X$ be a smooth Fano 3-fold with $\rho_X> 5$. Then $X\cong S\times\pr^1$ where $S$ is a del Pezzo surface. In particular $\rho_X\leq 10$.
\end{thm}

\medskip

The proof of Theorem \ref{main} relies on a careful study of {\em elementary contractions of $X$ of type $(3,2)$}, together with the author's previous work on Fano $4$-folds. 
To explain this, let us introduce some notation.

Let $X$ be a Fano $4$-fold. A  {\em contraction}  is a surjective morphism $f\colon X\to Y$, with connected fibers, where $Y$ is normal and projective;
$f$ is  {\em elementary} if $\rho_X-\rho_Y=1$. As usual, an elementary contraction can be of fiber type, divisorial, or small.

We say that an elementary contraction 
$f\colon X\to Y$ is {\em of type}
$(3,2)$ if it is divisorial with $\dim S=2$, where $E:=\Exc(f)$ and $S:=f(E)\subset Y$.
  Such $f$ can have at most finitely many $2$-dimensional fibers; outside the images of these fibers, $Y$ and $S$ are smooth, and $f$ is just the blow-up of the surface $S$.
 If $y_0\in S$ is the image of a two-dimensional fiber, then either $Y$ or $S$ are singular at $y_0$; these singularities have been described by Andreatta and Wi\'sniewski, see Theorem \ref{singularities}. In any case, $Y$ has at most isolated locally factorial and terminal singularities, while $S$ can be not normal.

We denote by $\mathcal{N}_{1}(X)$ the real vector space of one-cycles  with real coefficients, modulo numerical equivalence; we have $\dim\N(X)=\rho_X$. For any closed subset $Z\subset X$, we set
$$\N(Z,X):=\iota_*(\N(Z))\subset\N(X)$$ where $\iota\colon Z\hookrightarrow X$ is the inclusion, so that $\N(Z,X)$ is the subspace of $\N(X)$ spanned by classes of curves in $Z$, and $\dim\N(Z,X)\leq\rho_Z$.

We study an elementary contraction $f\colon X\to Y$ of type $(3,2)$ under the hypothesis that: $$\dim\N(E,X)\geq 4.$$ In particular this implies that $Y$ is Fano too (Lemma \ref{Fanotarget}).

We would like to compare $(-K_Y)_{|S}$ to $-K_S$, but since $S$ may be singular, 
we consider the minimal resolution of singularities $\mu\colon S'\to S$
and set $L:=\mu^*((-K_Y)_{|S})$, a nef and big divisor class on $S'$. We show that $K_{S'}+L$ is semiample (Proposition \ref{resolution}). Then
our strategy is to look for curves in $S'$ on which $K_{S'}+L$ is trivial, using other elementary contractions of $X$ of type $(3,2)$ whose exceptional divisor intersects $E$ in a suitable way.

Hence let us assume that $X$ has another elementary contraction $g_1$ of type $(3,2)$ whose exceptional divisor $E_1$ intersects $E$, and such that  $E\cdot \Gamma_1=0$ for a curve $\Gamma_1$ contracted by $g_1$. Set $D:=f(E_1)\subset Y$. We show that
an irreducible component $C_1$ of
$D\cap S$ is a $(-1)$-curve contained in the smooth locus $S_{\text{\em reg}}$, and such that $-K_Y\cdot C_1=1$ (Proposition \ref{oneray}, see Figure \ref{figura1} on p.~\pageref{figura1}). If $C_1'\subset S'$ is the transform of $C_1$, we have $(K_{S'}+L)\cdot C_1'=0$.

Finally let us assume that $X$ has three elementary contractions $g_1,g_2,g_3$, all of type $(3,2)$, satisfying the same assumptions as $g_1$ above. We also assume that $E_1\cdot\Gamma_2>0$ and $E_1\cdot \Gamma_3>0$, where $E_1=\Exc(g_1)$ and $\Gamma_2,\Gamma_3$ are curves contracted by $g_2,g_3$ respectively. 
Then we show that $S$ is a smooth del Pezzo surface with $-K_S=(-K_Y)_{|S}$ (Propositions \ref{tworays} and \ref{threerays}); let us give an overview of the proof.

The previous construction yields three distinct $(-1)$-curves $C_1',C_2',C_3'\subset S'$ such that $(K_{S'}+L)\cdot C_i'=0$ and $C_1'$ intersects both $C_2'$ and $C_3'$. This shows that the contraction of $S'$ given by $K_{S'}+L$  cannot be birational, namely $K_{S'}+L$ is not big. We also rule out the possibility of a contraction onto a curve, and conclude that $K_{S'}+L\equiv 0$. Finally
 we show that $\omega_S\cong \ol_Y(K_Y)_{|S}$, where $\omega_S$ is the dualizing  sheaf of $S$, and conclude that $S$ is  smooth and del Pezzo.

 We believe that these results can be useful in the study of Fano $4$-folds besides their use in the present work. It would be interesting to generalize this technique to higher dimensions.

\medskip

Let us now explain how we use these results to prove Theorem \ref{main}.
 We define the {\em Lefschetz defect} of $X$ as:
 $$\delta_X:=\max\bigl\{\codim\N(D,X)\,|\,D\subset X\text{ a prime divisor}\bigr\}.$$
 This invariant, introduced in \cite{codim}, measures the difference between the Picard number of $X$ and that of its prime divisors; we refer the  reader to \cite{rendiconti} for a survey on $\delta_X$.

Fano $4$-folds with $\delta_X\geq 3$ are classified, as follows.
\begin{thm}[\cite{codim}, Theorem 3.3]\label{codim}
Let $X$ be a smooth Fano $4$-fold. If $\delta_X\geq 4$, then $X\cong S_1\times S_2$ where $S_i$ are del Pezzo surfaces, and $\delta_X=\max_{i}\rho_{S_i}-1$.
\end{thm}
\begin{thm}[\cite{delta3}, Proposition 1.5]\label{delta3}
  Smooth Fano $4$-folds with $\delta_X=3$ are classified. They have $5\leq\rho_X\leq 8$, and if $\rho_X\in\{7,8\}$ then $X$ is a product of surfaces.
\end{thm}  
Therefore in our study of Fano $4$-folds we can assume that 
$\delta_X\leq 2$, that is, $\codim\N(D,X)\leq 2$ for every prime divisor $D\subset X$. To prove that $\rho_X\leq 12$,
 we look for a prime divisor $D\subset X$ with $\dim\N(D,X)\leq 10$.

To produce such a divisor, we look at contractions of $X$.
If $X$ has an elementary contraction of fiber type, or a  divisorial elementary contraction 
$f\colon X\to Y$ with $\dim f(\Exc(f))\leq 1$, it is not difficult to
find a prime divisor $D\subset X$
such that $\dim\N(D,X)\leq 3$, hence $\rho_X\leq 5$ (Lemmas  \ref{paris} and \ref{sabri}).

 The case where $X$ has a small elementary contraction is much harder and is treated in \cite{small}, where the following result is proven. 
    \begin{thm}[\cite{small}, Theorem 1.1]\label{small}
      Let $X$ be a smooth Fano 4-fold. If $X$  has a small elementary contraction, then $\rho_X\leq 12$.
  \end{thm}

  We are left with the case where every elementary contraction $f\colon X\to Y$ is of type $(3,2)$. 
In this case we show  (Theorem \ref{delpezzo}) that, if $\rho_X\geq 8$, we can apply our previous study of elementary contractions of type $(3,2)$, so that if $E:=\Exc(f)$ and $S:=f(E)\subset Y$, then 
$S$ is a smooth del Pezzo surface. This implies that $\dim\N(S,Y)\leq\rho_S\leq 9$,  $\dim\N(E,X)=\dim\N(S,Y)+1\leq 10$, and finally that $\rho_X\leq 12$,
proving Theorem \ref{main}. 

\medskip

The structure of the paper is as follows. In \S \ref{prel} we gather some preliminary results. Then in \S \ref{dp} we develop our study of elementary contractions of type $(3,2)$, while in \S \ref{final} we prove Theorem \ref{main}. 
\subsection{Notation}\label{notation}
\noindent We work over the field of complex numbers.

We will frequently use the definitions and apply the techniques of birational geometry and the Minimal Model Program, without explicit references. We refer the reader to \cite{debarreUT,matsuki,kollarmori} for background
and details.

Let $X$ be a  projective variety. 

 We denote by $\mathcal{N}_{1}(X)$ (respectively, $\mathcal{N}^{1}(X)$) the real vector space of one-cycles (respectively, Cartier divisors) with real coefficients, modulo numerical equivalence;
 $\dim \mathcal{N}_{1}(X)=\dim \mathcal{N}^{1}(X)=\rho_{X}$ is the Picard number of $X$.

 For any closed subset $Z\subset X$, we denote by $\N(Z,X)$ the subspace of $\N(X)$ spanned by classes of curves in $Z$.

  Let $C$ be a one-cycle of $X$, and $D$ a Cartier divisor. We denote by $[C]$ (respectively, $[D]$) the numerical equivalence class in $\mathcal{N}_{1}(X)$ (respectively, $\mathcal{N}^{1}(X)$). We also denote by $D^{\perp}\subset\N(X)$ the orthogonal hyperplane to the class $[D]$. 

  The symbol $\equiv$ stands for numerical equivalence (for both one-cycles and divisors), and $\sim$ stands for linear equivalence of divisors.

  $\operatorname{NE}(X)\subset \mathcal{N}_{1}(X)$ is the convex cone generated by classes of effective curves, and $\overline{\NE}(X)$ is its closure.  An \emph{extremal ray} $R$ is a 
  one-dimensional face of $\overline{\NE}(X)$. If $D$ is a Cartier divisor in $X$, we write $D\cdot R>0$, $D\cdot R=0$, and so on, if $D\cdot \gamma>0$, $D\cdot \gamma=0$, and so on, for a non-zero class $\gamma\in R$. We say that $R$ is $K$-negative if $K_X\cdot R<0$.

  A \emph{contraction} is a surjective morphism, with connected fibers, between normal projective varieties.
  
Suppose that
$X$ has terminal and locally factorial singularities, and is Fano. Then $\NE(X)$ is a convex polyhedral cone.
Given a 
contraction $f\colon X\to Y$, we denote by $\text{NE}(f)$ the convex subcone of $\text{NE}(X)$ generated by classes of curves contracted by $f$; we recall that there is a bijection between contractions of $X$ and faces of $\NE(X)$, given by $f\mapsto\NE(f)$. Moreover $\dim\NE(f)=\rho_X-\rho_Y$, in particular $f$ is \emph{elementary} (that is, $\rho_X-\rho_Y=1$) if and only if $\NE(f)$ is an extremal ray.

When $\dim X=4$, we say that an extremal ray $R$ is of type $(3,2)$ if the associated elementary contraction $f$ is of type $(3,2)$, namely if $f$ is divisorial with $\dim f(\Exc(f))=2$. We also set
 $E_R:=\Exc(f)$
 and denote by  $C_R\subset E_R$ a general fiber of $f_{|E_R}$; note that $E_R\cdot C_R=-1$ and $-K_X\cdot C_R=1$.

We will also consider the cones
 $\Eff(X)\subset\Nu(X)$ of classes of effective divisors, and $\mov(X)\subset\N(X)$ of classes of curves moving in a family covering $X$.
 Since $X$ is Fano,  both cones are  polyhedral; we have the duality relation $\Eff(X)=\mov(X)^{\vee}$.

 If $\ma{N}$ is a real vector space and $S\subset\ma{N}$ is a subset, we denote by $\R S$ the linear span of $S$.
 \section{Preliminaries}\label{prel}
 \noindent In this section we gather some preliminary results that will be used in the sequel.

 Andreatta and Wi\'sniewski have classified the possible $2$-dimensional fibers of an elementary contraction of type $(3,2)$ of a smooth Fano $4$-fold. In doing this, they also describe precisely the singularities both of the target, and of the image of the exceptional divisor, as follows.
    \begin{thm}[\cite{AW}, Theorem on p.\ 256]\label{singularities}
  Let $X$ be a smooth Fano $4$-fold and
  $f\colon X\to Y$ an elementary contraction
 of type $(3,2)$.
  Set $S:=f(\Exc(f))$.
  
 Then
$f$ can have at most finitely many $2$-dimensional fibers. Outside the images of these fibers, $Y$ and $S$ are smooth, and $f$ is the blow-up of $S$.

Let $y_0\in S\subset Y$ be the image of a $2$-dimensional fiber; then one of the following holds:
   \begin{enumerate}[$(i)$]
\item $S$ is smooth at $y_0$, while
  $Y$ has an ordinary double point at $y_0$, locally factorial and terminal;
  \item $Y$ is smooth at $y_0$, while
    $S$ is singular at $y_0$. More precisely either $S$ is not normal at $y_0$, or it has a singularity of type $\frac{1}{3}(1,1)$ at $y_0$ (as the cone over a twisted cubic).
  \end{enumerate}
  In particular  the singularities of $Y$ are at most isolated, locally factorial,  and terminal.
\end{thm}
We will need the following elementary estimates on $\dim\N(Z,X)$ in terms of a contraction $f\colon X\to Y$ and of $f(Z)\subset Y$.
\begin{remark}\label{pushforward}
  Let $f\colon X\to Y$ be a contraction between normal projective varieties, and  $Z\subset X$ an irreducible closed subset. Consider the pushforward of one-cycles $f_*\colon \N(X)\to\N(Y)$. We have the following:
  \begin{enumerate}[$(a)$]
    \item
      $f_*(\N(Z,X))=\N(f(Z),Y)$;
      \item $\dim\N(Z,X)\leq\rho_X-\rho_Y+\dim\N(f(Z),Y)$;
\item      
if $\dim f(Z)\leq 1$, then $\dim\N(Z,X)\leq \rho_X-\rho_Y+1$.
\end{enumerate}

Indeed $(a)$ follows from the definitions and the surjectivity of $f$, and $(b)$ follows from $(a)$ because
  $f_*$ is a surjective linear map. For $(c)$, we have
  $\N(f(Z),Y)=\{0\}$ if $f(Z)=\{pt\}$, and  $\N(f(Z),Y)=\R[f(Z)]$ if $f(Z)$ is a curve; in any case
 $\dim\N(f(Z),Y)\leq 1$, and we apply $(b)$.
\end{remark}
Now we give some simple preliminary results on extremal rays of type $(3,2)$.
\begin{lemma}\label{Fanotarget}
Let $X$ be a smooth Fano $4$-fold and $f\colon X\to Y$ an elementary  contraction of type $(3,2)$; set $E:=\Exc(f)$.
If $\dim\N(E,X)\geq 4$, then $E\cdot R\geq 0$
for every extremal ray $R$ of $X$ different from $\NE(f)$,
 and $Y$ is Fano.
\end{lemma}
\begin{proof}
It follows from \cite[Lemma 2.16 and Remark 2.17]{blowup}
that $\NE(f)$ is the unique extremal ray of $X$ having negative intersection with $E$,
 $-K_X+E=f^*(-K_Y)$ is nef, and $(-K_X+E)^{\perp}\cap\NE(X)=\NE(f)$, so that $-K_Y$ is ample.
\end{proof}
\begin{lemma}\label{fabri}
  Let $X$ be a smooth Fano $4$-fold and $R_1,R_2$  extremal rays of $X$ of type $(3,2)$  such that $\dim\N(E_{R_1},X)\geq 4$ and   $E_{R_1}\cdot R_2=0$.

  Then 
     $E_{R_2}\cdot R_1=0$ and $R_1+R_2$ is a face of $\NE(X)$ whose associated contraction is birational, with exceptional locus $E_{R_1}\cup E_{R_2}$.
\end{lemma}  
\begin{proof}  Let $H$ be a nef divisor on $X$ such that $H^{\perp}\cap\NE(X)=R_2$, and set $H':=H+(H\cdot C_{R_1})E_{R_1}$. Then
  $H'\cdot C_{R_1}= H'\cdot C_{R_2}=0$, and if $R_3$ is an extremal ray of $\NE(X)$ different from $R_1$ and $R_2$, we have $E_{R_1}\cdot R_3\geq 0$ by Lemma \ref{Fanotarget}, hence $H'\cdot R_3>0$. Therefore $H'$ is nef and $(H')^{\perp}\cap\NE(X)=R_1+R_2$ is a face of $\NE(X)$.

  If $\Gamma\subset X$ is an irreducible curve with $[\Gamma]\in R_1+R_2$, then
$H'\cdot\Gamma=0$, so that either $E_{R_1}\cdot\Gamma<0$ and $\Gamma\subset E_{R_1}$, or $H\cdot \Gamma=0$, $[\Gamma]\in R_2$ and $\Gamma\subset E_{R_2}$. This shows that 
the contraction of $R_1+R_2$ is birational with exceptional locus $E_{R_1}\cup E_{R_2}$.

\medskip

 We show that
 $E_{R_2}\cdot R_1=0$.
 By contradiction, suppose that  $E_{R_2}\cdot R_1\neq 0$.
 If  $E_{R_2}\cdot R_1<0$, then $E_{R_1}=E_{R_2}$, thus $\dim\N(E_{R_2},X)\geq 4$, contradicting Lemma \ref{Fanotarget}.

 Suppose that 
 $E_{R_2}\cdot R_1>0$, and
 let $f_i$ be the contraction of $R_i$, $i=1,2$.
 Since   $E_{R_2}\cdot R_1>0$, $E_{R_2}$ meets every non-trivial fiber of $f_1$,
 and $f_1(E_{R_1}\cap E_{R_2})=f_1(E_{R_1})$; let $Z$ be
 an irreducible component of $E_{R_1}\cap E_{R_2}$ such that $f_1(Z)=f_1(E_{R_1})$.

 On the other hand
 $E_{R_1}\cdot R_2=0$, thus $E_{R_1}\cap E_{R_2}$ is a union of fibers of $f_2$, and $\dim f_2(Z)\leq 1$. 
This yields
 $\dim\N(Z,X)\leq 2$ by Remark \ref{pushforward}$(c)$.

 We also have $f_1(Z)=f_1(E_{R_1})$, thus 
 $(f_1)_*(\N(E_{R_1},X))=(f_1)_*(\N(Z,X))$ by Remark \ref{pushforward}$(a)$, and
$\dim (f_1)_*(\N(E_{R_1},X))\leq \dim\N(Z,X)\leq 2$.
 We deduce that
 $\dim \N(E_{R_1},X)\leq 3$  by Remark \ref{pushforward}$(b)$, against our assumptions.
\end{proof}
\begin{lemma}\label{chitarra}
  Let $X$ be a smooth Fano $4$-fold and $R_1,R_2$ distinct extremal rays of $X$ of type $(3,2)$  with $\dim\N(E_{R_i},X)\geq 4$ for $i=1,2$. If there exists a birational contraction $g\colon X\to Z$ with $R_1,R_2\subset\NE(g)$,
  then
  $E_{R_1}\cdot R_2=E_{R_2}\cdot R_1=0$.  
\end{lemma}
\begin{proof}
  We note first of all that $E_{R_i}\cdot R_j\geq 0$ for $i\neq j$ by 
 Lemma \ref{Fanotarget}.
 Suppose that $E_{R_1}\cdot R_2>0$. Then $E_{R_1}\cdot (C_{R_1}+C_{R_2})=E_{R_1}\cdot C_{R_2}-1\geq 0$. Moreover 
 $E_{R_2}\cdot R_1>0$ by Lemma \ref{fabri}, so that $E_{R_2}\cdot (C_{R_1}+C_{R_2})\geq 0$.
 On the other hand for every prime divisor $D$ different from $E_{R_1},E_{R_2}$ we have $D\cdot (C_{R_1}+C_{R_2})\geq 0$, therefore $[C_{R_1}+C_{R_2}]\in\Eff(X)^{\vee}=\mov(X)$. Since $[C_{R_1}+C_{R_2}]\in\NE(g)$, $g$ should be of fiber type, a contradiction.
\end{proof} 
\begin{lemma}\label{paris}
  Let $X$ be a smooth Fano $4$-fold
with  $\delta_X\leq 2$, 
  and $g\colon X\to Z$ a contraction
 of fiber type. Then $\rho_Z\leq 4$.
 \end{lemma}
 \begin{proof}
   This follows from \cite{codim}; for the reader's convenience we report the proof.
   
   If $\dim Z\leq 1$, then $\rho_Z\leq 1$. If $Z$ is a surface, take any prime divisor $D\subset X$ such that $g(D)\subsetneq Z$, namely $\dim g(D)\leq 1$.   
We have
$\dim\N(D,X)\leq \rho_X-\rho_Z+1$ by Remark \ref{pushforward}$(c)$, thus
$\codim\N(D,X)\geq \rho_Z-1$. Therefore $\delta_X\leq 2$ yields $\rho_Z\leq 3$.

Suppose now that
$\dim Z=3$. By \cite[Lemma 2.6]{fanos}
 we know that $Z$ has some
elementary contraction  $h\colon Z\to W$. If $\dim W\leq 2$, by applying the first part of the proof to $h\circ g\colon X\to W$, we get $\rho_W\leq 3$ and hence
$\rho_Z\leq 4$.

If $h$ is birational and divisorial, then  $\dim h(\Exc(h))\leq 1$,
and  Remark \ref{pushforward}$(c)$ yields
$\dim\N(\Exc(h),Z)\leq 2$. Moreover we can take a prime divisor  $D\subset X$ such that
$g(D)\subseteq\Exc(h)$, thus $\dim\N(g(D),Z)\leq 2$. Reasoning as above we conclude that $\codim\N(D,X)\geq \rho_Z-2$ and $\rho_Z\leq 4$.

Finally we assume that $h$ is birational and small. Then $\Exc(h)$ is a curve in $Z$, and $\N(\Exc(h),Z)=\R\NE(h)$ is one-dimensional.
We show that there exists a prime divisor $D\subset X$ such that
$g(D)\subseteq\Exc(h)$.  
We consider the lifting of $h$
 in $X$ (see
 \cite[\S 2.5]{fanos}), which is an elementary contraction $h'\colon X\to W'$ fitting into a commutative diagram:
 $$\xymatrix{
   X\ar[d]_g\ar[r]^{h'}&{W'}\ar[d]\\
   Z\ar[r]^h&W
   }$$
and such that $g_*(\NE(h'))=\NE(h)$. If $F\subset X$ is a non-trivial fiber of $h'$, then $g$ must be finite on $F$ and $g(F)\subseteq\Exc(h)$. This implies that 
 $h'$ is a $K$-negative birational elementary contraction
with fibers of dimension $\leq 1$, therefore it must be
 of type $(3,2)$ (see \cite[Theorem 1.2]{wisn}); let $D$ be its exceptional
 divisor. Then $g(D)\subseteq\Exc(h)$.

Hence
 $\dim\N(g(D),Z)=1$, and reasoning as above we get $\codim\N(D,X)\geq \rho_Z-1$ and $\rho_Z\leq 3$.
\end{proof}
   \begin{lemma}[\cite{blowup}, Remark 2.17(1)]\label{sabri}
    Let $X$ be a smooth Fano $4$-fold. If $X$ has a divisorial elementary  contraction not of type $(3,2)$, then $\rho_X\leq 5$.
  \end{lemma}
\section{Showing that $S$ is a del Pezzo surface}\label{dp}  
\noindent In this section we study elementary contractions of type $(3,2)$ of a Fano $4$-fold. We focus on the surface $S$ which is the image of the exceptional divisor; as explained in the Introduction, our goal is to show that under suitable assumptions, $S$ is a smooth del Pezzo surface.

Recall that $S$ has isolated singularities by Theorem \ref{singularities}.
\begin{proposition}\label{resolution}
 Let $X$ be a smooth Fano $4$-fold and $f\colon X\to Y$ an elementary  contraction of type $(3,2)$.
Set $E:=\Exc(f)$ and $S:=f(E)$, and assume that $\dim\N(E,X)\geq 4$.

Let $\mu\colon S'\to S$ be the minimal resolution of singularities, and set $L:=\mu^*((-K_Y)_{|S})$. Then $K_{S'}+L$ is semiample.

If moreover $K_{S'}+L\equiv 0$, then  $S$ is a smooth del Pezzo surface, and $-K_S=(-K_Y)_{|S}$.
\end{proposition}
\begin{proof}
Note that $-K_Y$ is Cartier by Theorem \ref{singularities}, and ample by Lemma \ref{Fanotarget}, 
so that $L$ is nef and big on $S'$, and for every irreducible curve $\Gamma\subset S'$, we have $L\cdot\Gamma=0$ if and only if $\Gamma$ is $\mu$-exceptional.

Consider the pushforward of one-cycles $f_*\colon \N(X)\to\N(Y)$. Then $f_*(\N(E,X))=\N(S,Y)$, therefore $\rho_{S'}\geq\rho_S\geq\dim\N(S,Y)\geq 3$.

Recall that by the Cone Theorem we have:
$$\overline{\NE}(S')=\overline{\NE}(S')_{K_{S'}\geq 0}+\sum_i R_i $$
where $R_i$ are  the $K_{S'}$-negative extremal rays of $\overline{\NE}(S')$ (and they are at most countably many). 
We show that $K_{S'}+L$ is nef; for this it is enough to show that it is non-negative on each summand.

Since $L$ is nef, if $\gamma\in\overline{\NE}(S')_{K_{S'}\geq 0}$, we have
$(K_{S'}+L)\cdot\gamma=K_{S'}\cdot\gamma+L\cdot\gamma\geq 0$.

Suppose now that $\overline{\NE}(S')$ has  a $K_{S'}$-negative extremal ray
$R$. The contraction associated to $R$ can be onto a point (if $S'\cong\pr^2$), onto a curve (so that $\rho_{S'}=2$), or the blow-up of a smooth point
(see for instance \cite[Theorem 1-4-8]{matsuki}). Since
$\rho_{S'}>2$, $R$ 
 is generated by the class of a $(-1)$-curve $\Gamma$, that cannot be
  $\mu$-exceptional,  because $\mu$ is minimal. Then $L\cdot\Gamma>0$ and $(K_{S'}+L)\cdot\Gamma=L\cdot\Gamma-1\geq 0$.

  We conclude that $K_{S'}+L$ is nef on $S'$, and also semiample by the Base-Point-Free Theorem.

\medskip
  
  We assume now that
$K_{S'}+L\equiv 0$.
In particular $-K_{S'}$ is nef and big, namely $S'$ is a weak del Pezzo surface.

Set for simplicity $\ma{F}:=\ol_{Y}(K_{Y})_{|S}$, invertible sheaf on $S$, and let $\omega_S$ be the dualizing sheaf of $S$.
We have
 $K_{S'}\equiv \mu^*(\ma{F})$, and since $S'$ is rational, we also have $\ol_{S'}(K_{S'})\cong\mu^*(\ma{F})$.
 By restricting to the open subset $\mu^{-1}(S_{\text{\em reg}})$, we conclude that $(\omega_S)_{|S_{\text{\em reg}}}\cong \ma{F}_{|S_{\text{\em reg}}}$.
 Now we use the following. 
 \begin{lemma}\label{dualizing}
Let $S$ be a reduced and irreducible projective surface with isolated singularities, and $\omega_S$ its dualizing sheaf. If there exists an invertible sheaf $\ma{F}$ on $S$ such that $(\omega_S)_{|S_{\text{reg}}}\cong \ma{F}_{|S_{\text{reg}}}$, then $S$ is normal and $\omega_S\cong \ma{F}$.
\end{lemma}
This should be well-known to experts, we include a proof for lack of references.
 We postpone the proof of Lemma \ref{dualizing} and carry on with the proof of Proposition \ref{resolution}.
 
By Lemma \ref{dualizing} we have that $S$ is normal and  $\omega_S\cong \ma{F}$, in particular $\omega_S$ is locally free. If $y_0$ is a singular point of $S$, then 
 by Theorem \ref{singularities} $y_0$ is a singularity of type $\frac{1}{3}(1,1)$, but this contradicts the fact that  $\omega_S$ is locally free. We conclude that $S$ is smooth, and finally that $-K_S=(-K_{Y})_{|S}$ is ample, so that $S$ is a del Pezzo surface.
\end{proof}
\begin{remark}
  In the setting of Proposition \ref{resolution}, when $K_{S'}+L\equiv 0$ we cannot conclude that $Y$ is smooth. A priori $Y$ could have isolated singularities at some $y_0\in S$; by \cite{AW} in this case  $f^{-1}(y_0)\cong\pr^2$.
\end{remark}
\begin{proof}[Proof of Lemma \ref{dualizing}]
  Recall that $S$ has isolated singularities. 
The surface $S$ is reduced, thus it satisfies condition ($S_1$), namely $$\depth\ol_{S,y}\geq 1\quad\text{for every }y\in S.$$ Then by \cite[Lemma 1.3]{hart2007} the dualizing sheaf $\omega_S$ satisfies condition 
   ($S_2$):
   $$\depth\omega_{S,y}\geq 2\quad\text{for every }y\in S,$$
   where $\depth\omega_{S,y}$ is the depth of the stalk $\omega_{S,y}$ as an $\ol_{S,y}$-module.
   
 Then, for every open subset $U\subset S$ such that $S\smallsetminus U$ is finite, we have $\omega_S=j_*((\omega_S)_{|U})$, where $j\colon U\hookrightarrow S$ is the inclusion. This is analogous to the properties of reflexive sheaves on normal varieties, see \cite[Propositions 1.3 and 1.6]{hartrefl} and \cite[Remark 1.8]{hart2007}; for the reader's convenience, we recall
 the proof using local cohomology \cite{localcohomology}.
   
Set $\{y_1,\dotsc,y_m\}:=S\smallsetminus U$.
We have $\depth_{\{y_1,\dotsc,y_m\}}\omega_S:=\min_i\depth \omega_{S,y_i}\geq 2$ \cite[p.~43]{localcohomology}. By \cite[Theorem 3.8]{localcohomology} this is equivalent to  $\underline{H}^i_{\{y_1,\dotsc,y_m\}}(\omega_S)=0$ for $i=0,1$, where $\underline{H}^i_{\{y_1,\dotsc,y_m\}}(\omega_S)$ is the $i$th local cohomology sheaf of $S$ with coefficients in $\omega_S$ and supports in $\{y_1,\dotsc,y_m\}$ \cite[\S 1]{localcohomology}, in particular
 $\underline{H}^0_{\{y_1,\dotsc,y_m\}}(\omega_S)$ is the subsheaf of $\omega_S$ of sections with support contained in $\{y_1,\dotsc,y_m\}$.
There is an exact sequence of sheaves:
$$0\la \underline{H}^0_{\{y_1,\dotsc,y_m\}}(\omega_S)\la \omega_S \la j_*\bigl((\omega_S)_{|U}\bigr)\la
\underline{H}^1_{\{y_1,\dotsc,y_m\}}(\omega_S)\la 0$$
\cite[Corollary 1.9]{localcohomology},
hence $\underline{H}^i_{\{y_1,\dotsc,y_m\}}(\omega_S)=0$ for $i=0,1$
is in turn equivalent to 
$\omega_S=j_*((\omega_S)_{|U})$.

For $U=S_{\text{\em reg}}$ we have 
$\omega_S=j_*((\omega_S)_{|S_{\text{\em reg}}})$. Since 
 $\ma{F}$ is locally free, we get
$$\omega_S=j_*\bigl((\omega_S)_{|S_{\text{\em reg}}}\bigr)\cong j_*( \ma{F}_{|S_{\text{\em reg}}} )=\ma{F},$$
in particular $\omega_S$ is an invertible sheaf and for every $y\in Y$ we have $\omega_{S,y}\cong\ol_{S,y}$ as an $\ol_{S,y}$-module, thus $\depth\ol_{S,y}=2$. Therefore  $S$ has property $(S_2)$, and  it is normal by Serre's criterion.
\end{proof}
\begin{proposition}\label{oneray}
 Let $X$ be a smooth Fano $4$-fold and $f\colon X\to Y$ an elementary  contraction of type $(3,2)$.
Set $E:=\Exc(f)$ and $S:=f(E)$, and assume that $\dim\N(E,X)\geq 4$.
Let $\mu\colon S'\to S$ be the minimal resolution of singularities, and set $L:=\mu^*((-K_Y)_{|S})$.

  Suppose that $X$ has an extremal ray $R_1$ of type $(3,2)$ such that:
  $$E\cdot R_1=0\quad\text{and}\quad E\cap E_{R_1}\neq\emptyset.$$ Set $D:=f(E_{R_1})\subset Y$.
  
  Then $D_{|S}=C_1+\cdots+C_r$ where $C_i$ are pairwise disjoint $(-1)$-curves contained in $S_{\text{reg}}$, $E_{R_1}=f^*(D)$, and $f_*(C_{R_1})\equiv_Y C_i$.
Moreover  if $C_i'\subset S'$ is the transform of $C_i$, we have $(K_{S'}+L)\cdot C_i'=0$ for every $i=1,\dotsc,r$.
\end{proposition}
\stepcounter{thm}
\begin{figure}\caption{The varieties in Proposition \ref{oneray}.}\label{figura1}

\bigskip
  
  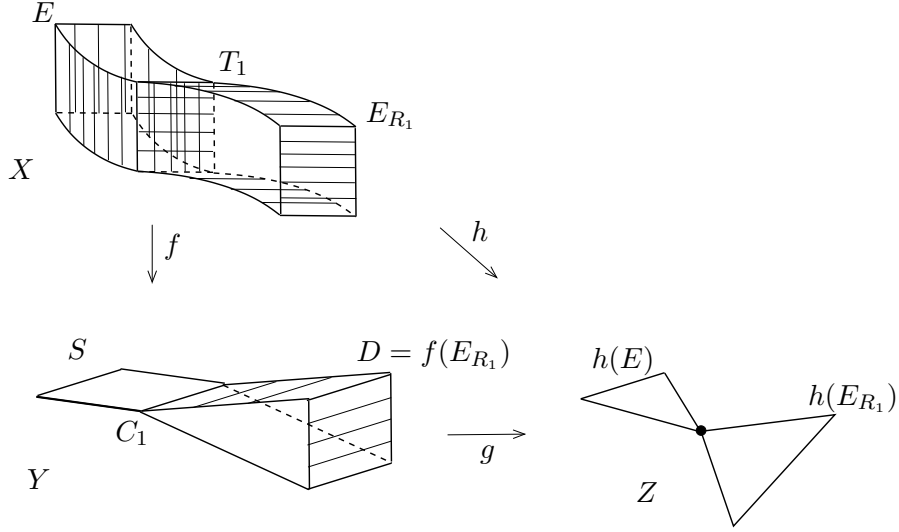
\end{figure}
\begin{proof}
  By Lemma \ref{fabri} we have
  $E_{R_1}\cdot \NE(f)=0$ and $\NE(f)+R_1$ is
 a face of $\NE(X)$, whose associated   contraction $h\colon X\to Z$ is birational with $\Exc(h)=E\cup E_{R_1}$.
 We have a diagram (see Figure \ref{figura1}):
 \stepcounter{thm}
 \begin{equation}\label{g}
  \xymatrix{ X\ar[d]_f\ar[dr]^h&\\
    Y\ar[r]_g&Z}
  \end{equation}
where $g$ is an elementary, $K$-negative, divisorial contraction, with $\Exc(g)=D$ (recall that $Y$ is  is locally factorial by Theorem \ref{singularities}, and Fano by Lemma \ref{Fanotarget}).

Since $E_{R_1}\cdot \NE(f)=E\cdot R_1=0$, $E\cap E_{R_1}$ is both a union of fibers of $f$ and of fibers of the contraction of $R_1$. This implies that
$\dim f(E\cap E_{R_1})\leq 1$, that $\N(E\cap E_{R_1})=\ker f_* \oplus\R R_1=\ker h_*$, and that $\dim h(E\cap E_{R_1})=0$.

We also note that both $E$ and $E_{R_1}$ have non-positive intersection  with every irreducible curve contracted by $h$, thus they are both unions of fibers of $h$, and
$$E\cap E_{R_1}=h^{-1}\bigl(h(E)\bigr)\cap h^{-1}\bigl(h(E_{R_1})\bigr)=h^{-1}\bigl(h(E)\cap h(E_{R_1})\bigr).$$ 
Both $h(E)$ and $h(E_{R_1})$ are surfaces in $Z$, and the general fiber of $h$ over these surfaces is one-dimensional. Moreover $h(E)\cap h(E_{R_1})=h(E\cap E_{R_1})$ is finite, and the connected components of  $E\cap E_{R_1}$ are $2$-dimensional fibers of $h$ over these points.

Using the classification of the possible $2$-dimensional fibers of $h$ in \cite{AW}, as in \cite[Lemma 4.15]{small} we see that every connected component $T_i$ of $E\cap E_{R_1}$ (which is non-empty by assumption) is isomorphic to $\pr^1\times\pr^1$ with normal bundle $\ol(-1,0)\oplus\ol(0,-1)$, for $i=1,\dotsc,r$. Set $C_i:=f(T_i)$, so that $D\cap S=f(E\cap E_{R_1})=f(\cup_i T_i)=\cup_iC_i$. Then
$C_i\cong\pr^1$, $C_i\cap C_j=\emptyset$ if $i\neq j$, and $f$ has fibers of dimension one over $C_i$, therefore $C_i\subset S_{\text{\em reg}}$ and $C_i\subset Y_{\text{\em reg}}$ by Theorem \ref{singularities}.

Moreover $g(D)=h(E_{R_1})$ is a surface, namely $g$ is of type $(3,2)$, and $C_i$ is a one-dimensional fiber of $g$ contained in $Y_{\text{\em reg}}$, hence  $K_Y\cdot C_i=D\cdot C_i=-1$. We also have $E_{R_1}=f^*(D)$ and $f_*(C_{R_1})\equiv_Y C_i$.

Since $C_i\subset S_{\text{\em reg}}$, it is a Cartier divisor in $S$, and we can write 
$D_{|S}=m_1C_1+\cdots+m_rC_r$ with $m_i\in\Z_{>0}$ for every $i=1,\dotsc,r$.
In $S$ we have $C_i\cdot C_j=0$ for $i\neq j$, hence for  $i\in\{1,\dotsc,r\}$ we get
$$-1=D\cdot C_i=(m_1C_1+\cdots+m_rC_r)\cdot C_i=m_i C_i^2$$
and we conclude that $m_i=1$ and $C_i^2=-1$, so that $C_i$ is a $(-1)$-curve in $S$.

 Finally $-K_S\cdot C_i=-K_Y\cdot C_i=1$, hence if $C_i'\subset S'$ is the transform of $C_i$, we have $(K_{S'}+L)\cdot C_i'=0$.
\end{proof}  
\begin{corollary}\label{torino}
 Let $X$ be a smooth Fano $4$-fold and $f\colon X\to Y$ an elementary  contraction of type $(3,2)$.
Set $E:=\Exc(f)$, and assume that $\dim\N(E,X)\geq 4$.
  Suppose that $X$ has an extremal ray $R_1$ of type $(3,2)$ such that $E\cdot R_1=0$.

  Then $R_1':=f_*(R_1)$ is an extremal ray of $Y$ of type $(3,2)$, and $E_{R_1}=f^*(E_{R_1'})$.
\end{corollary}
\begin{proof}
  If $E\cap E_{R_1}\neq\emptyset$, we are in the setting of Proposition \ref{oneray}; consider the elementary contraction $g\colon Y\to Z$ as in \eqref{g}. Then $\NE(g)=f_*(R_1)=R_1'$ is an extremal ray of $Y$ of type $(3,2)$, and $f^*(E_{R_1'})=E_{R_1}$.

   If $E\cap E_{R_1}=\emptyset$, then we still have a diagram as \eqref{g}, where $g$ is locally isomorphic to the contraction of $R_1$ in $X$, and the statement is clear.
\end{proof}  
\begin{proposition}\label{tworays}
 Let $X$ be a smooth Fano $4$-fold and $f\colon X\to Y$ an elementary  contraction of type $(3,2)$.
 Set $E:=\Exc(f)$ and $S:=f(E)$, and assume that $\dim\N(E,X)\geq 4$.
  Let $\mu\colon S'\to S$ be the minimal resolution of singularities, and set $L:=\mu^*((-K_Y)_{|S})$.

Suppose that $X$ has two extremal rays
  $R_1,R_2$ of type $(3,2)$  such that:
  $$E_{R_1}\cdot R_2>0 \text{ and } E\cdot R_i=0,\ E\cap E_{R_i}\neq\emptyset\ \text{for } i=1,2.$$
 
  Then one of the following holds:
  \begin{enumerate}[$(i)$]
     \item
     $K_{S'}+L\equiv 0$;
     \item there is a contraction $g\colon S'\to B$ with $\dim B=1$ such that $\NE(g)=(K_{S'}+L)^{\perp}\cap\overline{\NE}(S')$, and $E_{R_1}\cdot C_{R_2}=E_{R_2}\cdot C_{R_1}=1$. 
\end{enumerate}
\end{proposition}
\begin{proof}
We apply Proposition \ref{oneray}  to $f,R_1$ and to $f,R_2$. Write
   $f(E_{R_1})_{|S}=C_1+\cdots+C_r$, and let $\Gamma_2$ be an irreducible component of $f(E_{R_2})_{|S}$, so that $C_1,\dotsc,C_r,\Gamma_2$ are $(-1)$-curves contained in $S_{\text{\em reg}}$, and $\Gamma_2\equiv f_*(C_{R_2})$. Then
\stepcounter{thm}
   \begin{equation}\label{formula}
  0<E_{R_1}\cdot C_{R_2}=f^*(f(E_{R_1}))\cdot C_{R_2}=f(E_{R_1})\cdot \Gamma_2=(C_1+\cdots+C_r)\cdot \Gamma_2,\end{equation}
 hence $C_i\cdot \Gamma_2>0$ for some $i$, say $i=1$. Since $C_1$ cannot be a component of $f(E_{R_2})_{|S}$, we also get $E_{R_2}\cdot C_{R_1}=f(E_{R_2})_{|S}\cdot C_1\geq \Gamma_2\cdot C_1>0$.

 Let  $\Gamma_2'$ and $C_1'$ in $S'$ be the transforms of $\Gamma_2$ and $C_1$ respectively;  then $\Gamma_2'$ and $C_1'$ are disjoint from the $\mu$-exceptional locus, are $(-1)$-curves in $S'$,
$C_1'\cdot \Gamma_2'>0$, and still by  Proposition \ref{oneray} we have
 $(K_{S'}+L)\cdot C_1'=(K_{S'}+L)\cdot\Gamma_2'=0$.

 Recall that  $K_{S'}+L$ is semiample by Proposition \ref{resolution}. In particular, the face $(K_{S'}+L)^{\perp}\cap\overline{\NE}(S')$ contains the classes of two distinct $(-1)$-curves which meet. This means that the associated contraction cannot be birational, and we have two possibilities: either $K_{S'}+L\equiv 0$, or  $K_{S'}+L$ yields a contraction $g\colon S'\to B$ onto a smooth curve. We show that this second case yields
$(ii)$.

 Let $F\subset S'$ be a general fiber $F$ of $g$, so that $-K_{S'}\cdot F=L\cdot F$. Since $F$
is not $\mu$-exceptional, we have $L\cdot F>0$ and hence 
$-K_{S'}\cdot F>0$. Thus there is a non-empty open subset $B_0\subseteq B$ such that $(-K_{S'})_{|g^{-1}(B_0)}$ is $g$-ample,  therefore
$g_{|g^{-1}(B_0)}\colon g^{-1}(B_0)\to B_0$ is a conic bundle,
$F\cong\pr^1$, and $-K_{S'}\cdot F=2$.

 The curves
 $C'_1$ and $\Gamma_2'$ are components of the same fiber $F_0$ of $g$, and $-K_{S'}\cdot F_0=2=-K_{S'}\cdot (C_1'+\Gamma_2')$.
 For any irreducible curve $C_0$ contained in $F_0$ we have $-K_{S'}\cdot C_0=L\cdot C_0\geq 0$, so that if $C_0$ is different from  $C_1'$ and $\Gamma_2'$, we must have $-K_{S'}\cdot C_0=L\cdot C_0= 0$ and $C_0$ is $\mu$-exceptional.  But $C_1'$ and $\Gamma_2'$ are disjoint from the $\mu$-exceptional locus, 
thus $C_0\cap (C_1'\cup\Gamma_2')=\emptyset$. Since $F_0$ is connected, we conclude that
  $F_0= C_1'+\Gamma_2'$ and $F_0\subset g^{-1}(B_0)$, hence $F_0$ is isomorphic to a reducible conic.

  This also shows that $C'_i$ for $i>1$ are  contained in different fibers of $g$, so that
$$ C_1\cdot \Gamma_2=\Gamma_2\cdot C_1=1\quad\text{and}\quad C_i\cdot \Gamma_2=0\quad \text{for every }i=2,\dotsc,r,$$ and finally using \eqref{formula}
$$E_{R_1}\cdot C_{R_2}=(C_1+\cdots+C_r)\cdot \Gamma_2=1.$$
Similarly we conclude that $E_{R_2}\cdot C_{R_1}=1$.
\end{proof}  
\begin{proposition}\label{threerays}
Let $X$ be a smooth Fano $4$-fold and $f\colon X\to Y$ an elementary  contraction of type $(3,2)$.
Set $E:=\Exc(f)$ and $S:=f(E)$, and assume that $\dim\N(E,X)\geq 4$.

Suppose that $X$ has three distinct extremal rays
  $R_1,R_2,R_3$ of type $(3,2)$ such that: 
  $$E\cdot R_i=0,\ E\cap E_{R_i}\neq\emptyset\ \text{for } i=1,2,3, \text{ and } E_{R_1}\cdot R_j>0\ \text{for } j=2,3.$$
  
  Then 
 $S$ is a smooth del Pezzo surface and $-K_S=(-K_Y)_{|S}$.
\end{proposition}  
\begin{proof}
  We apply Proposition \ref{tworays} to $f,R_1,R_2$ and to $f,R_1,R_3$; we show that we are in case $(i)$, which 
  yields the statement by Proposition \ref{resolution}.

  By contradiction, suppose that we are in case $(ii)$;
   we keep the same notation as in the proof of  Proposition \ref{tworays}.
  Then  $K_{S'}+L$ yields a contraction $g\colon S'\to B$ onto a curve,  $E_{R_2}\cdot R_1>0$, and $E_{R_3}\cdot R_1>0$. 
  
  Let $C_1\subset S$ be an irreducible component of $f(E_{R_1})_{|S}$, and $C_1'\subset S'$ its transform. For $j\in\{2,3\}$ write $f(E_{R_j})_{|S}=\Gamma_{j1}+\cdots+\Gamma_{jr_j}$, and let $\Gamma_{ji}'\subset S'$ be the transform of $\Gamma_{ji}$.

  Using \eqref{formula} as in the proof of Proposition \ref{tworays}, we see that $(\Gamma_{j1}+\cdots+\Gamma_{jr_j})\cdot C_1>0$, hence $\Gamma_{ja_j}\cdot C_1>0$ for some $a_j\in\{1,\dotsc,r_j\}$, and
  $\Gamma_{ja_j}'\cdot C_1'>0$ in $S'$. Then the proof of Proposition \ref{tworays} shows that $C_1'+\Gamma_{2a_2}'$ and $C_1'+\Gamma_{3a_3}'$ are both fibers of $g$, so they should coincide, but $\Gamma_{2a_2}'\neq\Gamma_{3a_3}'$ because $R_2\neq R_3$, and we get a contradiction.
\end{proof}
\begin{corollary}\label{fourrays}
  Let $X$ be a smooth Fano $4$-fold with $\delta_X\leq 2$.  Suppose that $X$ has four distinct extremal rays
  $R_0,R_1,R_2,R_3$ of type $(3,2)$ such that: 
  $$E_{R_0}\cdot R_i=0\  \text{for } i=1,2,3, \text{ and } E_{R_1}\cdot R_j>0\ \text{for } j=2,3.$$
  Then one of the following holds:
  \begin{enumerate}[$(i)$]
    \item $\dim\N(E_{R_i},X)\leq 3$ for some $i\in\{0,1,2,3\}$, in particular $\rho_X\leq 5$;
    \item  $\dim\N(E_{R_0},X)\leq 10$, in particular $\rho_X\leq 12$.

      \noindent Moreover if $f\colon X\to Y$ is the contraction of $R_0$ and $S:=f(E_{R_0})$, then 
 $S$ is a smooth del Pezzo surface and $-K_S=(-K_Y)_{|S}$. 
    \end{enumerate}
\end{corollary}  
\begin{proof}
We assume that  $\dim\N(E_{R_i},X)\geq 4$ for every $i=0,1,2,3$, and prove $(ii)$.

We show that 
$E_{R_0}\cap E_{R_i}\neq\emptyset$ for every $i=1,2,3$.
If  $E_{R_0}\cap E_{R_i}=\emptyset$ for some $i\in\{1,2,3\}$, then  for every curve $C\subset E_{R_0}$ we have $E_{R_i}\cdot C=0$, so that $[C]\in(E_{R_i})^{\perp}$, and $\N(E_{R_0},X)\subset(E_{R_i})^{\perp}$.
 
Since the classes $[E_{R_1}],[E_{R_2}],[E_{R_3}]\in\Nu(X)$ generate distinct one dimensional faces of $\Eff(X)$ (see \cite[Remark 2.19]{eff}), they are linearly independent, hence in $\N(X)$ we have
$$\codim \bigl((E_{R_1})^{\perp}\cap (E_{R_2})^{\perp}\cap (E_{R_3})^{\perp}\bigr)=3.$$
On the other hand $\codim\N(E_{R_0},X)\leq \delta_X\leq 2$, thus $\N(E_{R_0},X)$ cannot be contained in the above intersection. Then $\N(E_{R_0},X)\not\subset (E_{R_{h}})^{\perp}$ for some $h\in\{1,2,3\}$, hence $E_{R_0}\cap E_{R_h}\neq\emptyset$. 
In particular, since $E_{R_0}\cdot R_h=0$, there exists an irreducible curve $C\subset E_{R_0}$ with $[C]\in R_{h}$.

For $j=2,3$ we have  $E_{R_1}\cdot R_j>0$, and by Lemma \ref{fabri} also  $E_{R_j}\cdot R_1>0$. This implies that $E_{R_0}\cap E_{R_i}\neq\emptyset$ for every $i=1,2,3$. For instance say $h=3$: then $E_{R_1}\cdot R_3>0$ yields $E_{R_1}\cap C\neq\emptyset$, hence  $E_{R_0}\cap E_{R_1}\neq\emptyset$. Then there exists  an irreducible curve $C'\subset E_{R_0}$ with $[C']\in R_{1}$, and $E_{R_2}\cdot R_1>0$ yields 
 $E_{R_0}\cap E_{R_2}\neq\emptyset$.

Finally we apply  Proposition \ref{threerays} to get that
$S$ is a smooth del Pezzo surface and $-K_S=(-K_Y)_{|S}$. Therefore $\dim\N(S,Y)\leq\rho_S\leq 9$ and $\dim\N(E_{R_0},X)=\dim\N(S,X)+1\leq 10$, so we get $(ii)$.
\end{proof}
\section{Proof of Theorem \ref{main}}\label{final}
\noindent In this section we show how to apply the results of \S \ref{dp} to bound $\rho_X$; the following is our main result.  
\begin{thm}\label{delpezzo}
  Let $X$ be a smooth Fano $4$-fold  with $\delta_X\leq 2$ and $\rho_X\geq 8$, and with no small elementary contraction.

  Then $\rho_X\leq\delta_X+10\leq 12$.
  Moreover  every elementary contraction $f\colon X\to Y$ is of type $(3,2)$, and $S:=f(\Exc(f))\subset Y$ is a smooth del Pezzo surface with $-K_S=(-K_Y)_{|S}$.
\end{thm}
In the proof we will use the following terminology: if $R_1$, $R_2$ are distinct one-dimensional faces of a convex polyhedral cone $\ma{C}$, we say that $R_1$ and $R_2$ are \emph{adjacent} if $R_1+R_2$ is a face of $\ma{C}$.
A  \emph{facet} of $\ma{C}$ is a face of codimension one.
We will also need the following elementary fact.
\begin{lemma}\label{cone}
  Let $\ma{C}$ be a convex polyhedral cone not containing non-zero linear subspaces, and $R_0$ a one-dimensional face of $\ma{C}$. Let $R_1,\dotsc,R_m$ be the one-dimensional faces of $\ma{C}$ that are adjacent to $R_0$.

  Then the linear span of $R_0,R_1,\dotsc,R_m$ is $\R\ma{C}$.
\end{lemma}
\begin{proof}
  We can assume that $\ma{C}\subset\R^n$ with $n=\dim\ma{C}$. Since $\ma{C}$ does not contain non-zero linear subspaces, there exists
an affine hyperplane 
  $H\subset\R^n$ such that $P:=H\cap\ma{C}$ is an $(n-1)$-dimensional convex polytope, and $\ma{C}$ is the cone over $P$. Then $v_i:=R_i\cap H$ is a vertex of $P$ for $i=0,1,\dotsc,m$, and $v_1,\dotsc,v_m$
are the vertices of $P$
that are adjacent to $v_0$. The claim is that the affine span of $v_0,v_1,\dotsc,v_m$ is $H$.

Up to translation we can assume that $v_0=0$ in $H=\R^{n-1}$. Let $\ma{D}\subset H$ be the convex cone generated by $v_1,\dotsc,v_m$, with vertex $v_0$. Since $P$ is convex, we have $P\subset\ma{D}$, and $\dim\ma{D}=\dim P=n-1$. Thus the affine span of $v_0,v_1,\dotsc,v_m$  has dimension $n-1$, and coincides with $H$.
\end{proof}  
\begin{proof}[Proof of  Theorem \ref{delpezzo}]
   Let $f\colon X\to Y$ be an elementary contraction; note that $\rho_Y=\rho_X-1\geq 7$.
    Then $f$ is not of fiber type by Lemma \ref{paris}, and not small by assumption, so that $f$ is
    divisorial. Moreover $f$ is of type $(3,2)$ by Lemma \ref{sabri}.

Set $E:=\Exc(f)$ and $S:=f(E)\subset Y$; we have $\dim\N(E,X)\geq\rho_X-\delta_X\geq 6$, and  if $R'\neq \NE(f)$ is another extremal ray of $X$, we have $E\cdot R'\geq 0$ by Lemma \ref{Fanotarget}.
Moreover,
    if $R'$ is  adjacent to $\NE(f)$, then $E\cdot R'=0$. Indeed the contraction $g\colon X\to Z$ of the face $R'+\NE(f)$ cannot be of fiber type by Lemma \ref{paris}, thus it is birational and we apply Lemma \ref{chitarra}.

We are going to show that there exists three extremal rays $R_1',R_2',R'_3$ adjacent to $\NE(f)$ such that $E_{R_1'}\cdot R_j'>0$ for $j=2,3$, and then apply Corollary \ref{fourrays}.
    
\medskip
  
Let us consider the cone $\NE(Y)$. It is a convex polyhedral cone whose extremal rays $R$ are in bijection with the extremal rays $R'$ of $X$ adjacent to $\NE(f)$, via $R=f_*(R')$, see \cite[\S 2.5]{fanos}.

By Corollary \ref{torino}, $R$ is still of type $(3,2)$, and  $f^*(E_{R})=E_{R'}$. Thus for every pair $R_1,R_2$ of distinct extremal rays of $Y$, with $R_i=f_*(R_i')$ for $i=1,2$, we have $E_{R_1}\cdot R_2=E_{R'_1}\cdot R'_2\geq 0$.

If $R_1$ and $R_2$ are adjacent, we show that $E_{R_1}\cdot R_2=E_{R_2}\cdot R_1=0$. Indeed consider the contraction $Y\to Z$ of the face $R_1+R_2$ and the composition $g\colon X\to Z$, which contracts $R_1'$ and $R_2'$. Again $g$ cannot be of fiber type by Lemma \ref{paris}, thus it is birational and we apply Lemma \ref{chitarra} to get $E_{R_1'}\cdot R_2'=E_{R_2'}\cdot R_1'=0$, thus
$E_{R_1}\cdot R_2=E_{R_2}\cdot R_1=0$.

\medskip
  
Fix an extremal ray $R_1$ of $Y$.
We show that there exist two distinct extremal rays $R_2,R_3$ of $Y$ with $E_{R_1}\cdot R_j>0$ for $j=2,3$.

Indeed since $E_{R_1}$ is an effective divisor, there exists some curve $C\subset Y$ with $E_{R_1}\cdot C>0$, hence there exists some extremal ray $R_2$ with $E_{R_1}\cdot R_2>0$.

By contradiction, let us assume that $E_{R_1}\cdot R=0$ for every extremal ray $R$ of $Y$ different from $R_1,R_2$. This means that the cone $\NE(Y)$ has the extremal ray $R_1$ in the halfspace $\N(Y)_{E_{R_1}<0}$, the extremal ray $R_2$ in the halfspace $\N(Y)_{E_{R_1}>0}$, and all other extremal rays in the
 hyperplane $(E_{R_1})^{\perp}$.

Fix $R\neq R_1,R_2$, and let $\tau$ be a facet of $\NE(Y)$ containing $R$ and not $R_1$. Note that $\R\tau\neq (E_{R_1})^{\perp}$, as $E_{R_1}$ and $-E_{R_1}$ are not nef.
By Lemma \ref{cone}
the rays adjacent to $R$ in $\tau$ cannot be all contained in $(E_{R_1})^{\perp}$. We conclude that $R_2$ is adjacent to $R$, therefore $E_{R_2}\cdot R=0$, namely $R\subset (E_{R_2})^{\perp}$.

Summing up, we have shown that every extremal ray $R\neq R_1,R_2$ of $Y$ is contained in both $(E_{R_1})^{\perp}$ and $(E_{R_2})^{\perp}$. On the other hand these rays include all the rays adjacent to $R_1$, so by Lemma \ref{cone} their linear span must be at least a hyperplane. Therefore
$(E_{R_1})^{\perp}=(E_{R_2})^{\perp}$ and the classes $[E_{R_1}],[E_{R_2}]\in\Nu(Y)$ are proportional, which is impossible, because they generate distinct one dimensional faces of the cone $\Eff(Y)$ (see \cite[Remark 2.19]{eff}).

We conclude that there exist  two distinct extremal rays $R_2,R_3$ of $Y$ with $E_{R_1}\cdot R_j>0$ for $j=2,3$.

For $i=1,2,3$ we have $R_i=f_*(R'_i)$ where $R'_i$ is an extremal ray of $X$ adjacent to $\NE(f)$, so that $E\cdot R_i'=0$. Moreover for $j=2,3$ we have
$E_{R'_1}\cdot R'_j=E_{R_1}\cdot R_j>0$.

We apply  Corollary \ref{fourrays} to $\NE(f),R_1',R_2',R_3'$. We have already excluded $(i)$, and $(ii)$ yields the statement.
\end{proof}
We can finally prove the following more detailed version of Theorem \ref{main}.
\begin{thm}\label{last}
  Let $X$ be a smooth Fano $4$-fold which is not a product of surfaces.
  
  Then $\rho_X\leq 12$,
 and if $\rho_X=12$, then there exist $X\stackrel{\ph}{\dasharrow} X'\stackrel{g}{\to} Z$
where $\ph$ is a finite sequence of flips,  $X'$ is smooth, $g$ is a contraction, and $\dim Z=3$.
\end{thm}
\begin{proof}
Since $X$ is not a product of surfaces, we have $\delta_X\leq 3$ by Theorem \ref{codim}. Moreover $\delta_X=3$ yields $\rho_X\leq 6$ by 
Theorem \ref{delta3}, while $\delta_X\leq 2$ yields
 $\rho_X\leq 12$ by Theorems \ref{small} and
\ref{delpezzo}.

If $\rho_X=12$, the  statement follows from 
 \cite[Theorems 2.7 and 9.1]{small}. 
\end{proof}  

\smallskip

\noindent {\bf Acknowledgements.} I thank the referee for  comments that helped to improve the readability of the paper.

\small
\providecommand{\noop}[1]{}
\providecommand{\bysame}{\leavevmode\hbox to3em{\hrulefill}\thinspace}
\providecommand{\MR}{\relax\ifhmode\unskip\space\fi MR }
\providecommand{\MRhref}[2]{%
  \href{http://www.ams.org/mathscinet-getitem?mr=#1}{#2}
}
\providecommand{\href}[2]{#2}

\end{document}